\documentclass[12pt, reqno]{amsart}
\usepackage{amsmath, amsthm, amscd, amsfonts, amssymb, graphicx, color}
\usepackage[bookmarksnumbered, colorlinks, plainpages]{hyperref}
\hypersetup{colorlinks=true,linkcolor=red, anchorcolor=green, citecolor=cyan, urlcolor=red, filecolor=magenta, pdftoolbar=true}

\newtheorem{theorem}{Theorem}[section]
\newtheorem{lemma}[theorem]{Lemma}

\newtheorem{cor}[theorem]{Corollary}
\theoremstyle{definition}
\newtheorem{definition}[theorem]{Definition}

\theoremstyle{remark}
\newtheorem{remark}[theorem]{\bf{Remark}}
\numberwithin{equation}{section}
\begin{document}

\title [$A$-numerical radius : Inequalities and equalities  ]{{$A$-numerical radius : New inequalities and characterization of equalities}}

\author[P. Bhunia and K. Paul] {Pintu Bhunia and Kallol Paul}

\address{(Bhunia) Department of Mathematics, Jadavpur University, Kolkata 700032, West Bengal, India}
\email{pintubhunia5206@gmail.com; 	pbhunia.math.rs@jadavpuruniversity.in}

\address{(Paul) Department of Mathematics, Jadavpur University, Kolkata 700032, West Bengal, India}
\email{kalloldada@gmail.com; kallol.paul@jadavpuruniversity.in}

\noindent \thanks{First  author would like to thank UGC, Govt. of India for the financial support in the form of SRF}
\thanks{}
\thanks{}



\maketitle

\begin{abstract}
We develope new lower bounds for the $A$-numerical radius of semi-Hilbertian space operators, and applying these bounds we obtain upper bounds for the $A$-numerical radius of the commutators of operators. The bounds obtained here improve on the existing ones. Further, we provide   characterizations for the equality of the existing $A$-numerical radius inequalities of semi-Hilbertian space operators.
\end{abstract}

\section{Introduction}

\noindent
Let $\mathcal{B}(\mathcal{H})$ denote the $\mathbb{C}^*$-algebra of all bounded linear operators on a complex Hilbert space $\mathcal{H}$ with inner product $ \langle .,.\rangle $ and the corresponding norm $\|.\|.$ For $T\in \mathcal{B}(\mathcal{H}),$   the range  and the kernel of $T $ are denoted by $R(T)$ and $N(T)$, respectively. By $\overline{{R}(T)},$ we denote the norm closure of ${R}(T).$ The Hilbert-adjoint of $T$ is denoted by $T^*.$  Let $A \in \mathcal{B}(\mathcal{H})$ be a positive operator, henceforth we reserve the symbol $A$ for positive operator on $\mathcal{H}.$ Clearly, $A$ induces a positive semidefinite sesquilinear form $\langle . , . \rangle_A: \mathcal{H} \times \mathcal{H} \rightarrow \mathbb{C}$, defined by $ \langle x, y \rangle_A = \langle Ax, y \rangle$ for all $x,y\in \mathcal{H}.$ This sesquilinear form induces a seminorm $\|.\|_A: \mathcal{H}\rightarrow \mathbb{R}^+$, defined by $\|x\|_A = \sqrt{\langle x, x \rangle_A}$  for all $ x \in \mathcal{H}.$
Clearly, $\|.\|_A$ is a norm if and only if $A$ is injective, and $(\mathcal{H}, \|.\|_A)$ is complete if and only if $R(A)$ is closed in $ \mathcal{H}.$ An operator  $T \in  \mathcal{B}(\mathcal{H})$ is said to be $A$-bounded  if there exists $c>0$ such that $\|Tx\|_A \leq c\|x\|_A$ for all $x\in \overline{{R}(A)}$, and in this case
\[\|T\|_A= \sup_{x\in  \overline{{R}(A)},\,{x\neq 0}}   \frac{\|Tx\|_A}{\|x\|_A}=\sup_{x\in  \overline{{R}(A)},\,{\|x\|_A=1}}   {\|Tx\|_A}<+\infty.\] 
Let $\mathcal{B}^A(\mathcal{H})$ denote the collection of all $A$-bounded operators, i.e., 
$ \mathcal{B}^A(\mathcal{H})=\{ T\in \mathcal{B}(\mathcal{H}): \|T\|_A <+ \infty\}.$ It is well known that $\mathcal{B}^A(\mathcal{H})$ is, in general, not  a sub-algebra of $\mathcal{B}(\mathcal{H}).$ Note that  $\|T\|_A=0$ if and only if $ATA=0.$ The $A$-adjoint of $T$, if it exists,  is defined as follows:
\begin{definition}(\cite{Arias1})
For $T \in \mathcal{B}(\mathcal{H})$, an operator $S \in \mathcal{B}(\mathcal{H})$ is said to be an A-adjoint of $T$ if  $\langle Tx, y \rangle_A = \langle x, Sy \rangle_A$ for all $x,y\in \mathcal{H}$, i.e, $AS=T^*A$. 
\end{definition}
\noindent  Let $ \mathcal{B}_A(\mathcal{H})$ denote the collection of all operators in $\mathcal{B}(\mathcal{H})$, which admit A-adjoint.  By Douglas theorem \cite{Douglas}, it follows that 
$$ \mathcal{B}_A(\mathcal{H}) = \{ T \in \mathcal{B}(\mathcal{H}) : R(T^*A) \subseteq R(A) \}.$$
For $T\in \mathcal{B}_A(\mathcal{H})$,  the operator equation  $AX=T^*A$ has a unique solution, denoted by $T^{\sharp_A}$, satisfying ${R}(T^{\sharp_A})\subseteq \overline{{R}(A)}$. Note that $T^{\sharp_A}=A^\dag T^*A$, where $A^\dag$ is the Moore-Penrose inverse of $A$. 
We note that $ \mathcal{B}_A(\mathcal{H}) \left(\subseteq \mathcal{B}^A(\mathcal{H})\right)$  is  a sub-algebra of $ \mathcal{B}(\mathcal{H}) $.
For $T \in \mathcal{B}_A(\mathcal{H})  $,  we have,  $AT^{\sharp_A}= T^*A$ and $ N(T^{\sharp_A}) = N(T^*A).$ If $T \in  \mathcal{B}_A(\mathcal{H})$, then $T^{\sharp_A} \in  \mathcal{B}_A(\mathcal{H})$  and $(T^{\sharp_A})^{\sharp_A} = P_{\overline{R(A)}}TP_{\overline{R(A)}}$ where $P_{\overline{R(A)}}$ is the orthogonal projection onto ${\overline{R(A)}}$. 
	An operator $T\in \mathcal{B}_A(\mathcal{H})$ is said to be A-self-adjoint if $AT$ is self-adjoint, i.e., $AT=T^*A$.
For further study on the A-adjoint operator, we refer to \cite{Arias1}. Note that, for $T,S \in  \mathcal{B}_A(\mathcal{H})$,  $(TS)^{\sharp_A}=S^{\sharp_A}T^{\sharp_A}$, $\|TS\|_A\leq \|T\|_A\|S\|_A$ and $\|Tx\|_A\leq \|T\|_A\|x\|_A$ for all $x\in \mathcal{H}.$
Clearly, for $T \in \mathcal{B}_{A}(\mathcal{H})$, $\|TT^{\sharp_A}\|_A = \|T^{\sharp_A}T\|_A = \|T^{\sharp_A}\|^2_A = \|T\|^2_A. $ 

\noindent For $T \in \mathcal{B}_{A}(\mathcal{H})$,  the A-numerical range of $T$, denoted by $W_A(T)$, is defined as  $$W_A(T) = \{\langle Tx, x \rangle_A: x\in \mathcal{H}, \|x\|_A= 1 \}$$
and the $A$-numerical radius of $T$,  denoted by $w_A(T)$,   is defined as
\begin{eqnarray*}
w_A(T) &=& \sup \{|\langle Tx, x \rangle_A|: x\in \mathcal{H}, \|x\|_A= 1 \}.
\end{eqnarray*}
It was shown in \cite{Zamani} that for $T\in \mathcal{B}_A(\mathcal{H})$, $w_A(T)=\sup_{\theta\in  \mathbb{R}} \left \|\frac{ e^{{\rm i} \theta}T+(e^{{\rm i} \theta}T)^{\sharp_A} }{2} \right \|_A$.
For $T \in \mathcal{B}_{A}(\mathcal{H})$, we have $$\|T\|_A = \sup\{\|Tx\|_A:  \|x\|_A=1\}=\sup \{|\langle Tx, y \rangle_A|  : \|x\|_A=\|y\|_A= 1 \}.$$
\noindent It is well-known that $w_A(.)$ and $\|.\|_A$  are equivalent seminorms on $\mathcal{B}_{A}(\mathcal{H})$  satisfying  the inequality (see \cite[Prop. 2.5]{BFS})
\begin{eqnarray}\label{eqv}
 \frac{1}{2} \|T\|_A \leq w_A(T) \leq \|T\|_A.
\end{eqnarray}
The inequalities in (\ref{eqv}) are sharp (see \cite{Feki1}), $w_A(T)=\frac{1}{2} \|T\|_A $ if $AT^2=0$ and $ w_A(T) = \|T\|_A$ if $AT=T^*A$. An improvement of  (\ref{eqv}) is given  in  \cite{Feki2,Zamani}, which is 
\begin{eqnarray}\label{eqv1}
\frac{1}{4} \left \| T^{\sharp_A}T+TT^{\sharp_A} \right\|_A  \leq w_A^2(T) \leq \frac{1}{2}\left \| T^{\sharp_A}T+TT^{\sharp_A} \right\|_A.
\end{eqnarray}
More refinements in this direction are also given in \cite{P5,P1,P2,P3,P4,Mos,Rout}. Inspired by the inequalities obtained in the article \cite{bp1,P10}, here we obtain new refinements of the first inequalities in (\ref{eqv}) and (\ref{eqv1}). By applying the new refinements  developed here, we obtain upper bounds for the $A$-numerical radius of the commutators of operators. Also, we obtain  characterizations for the equality of first inequalities in  (\ref{eqv}) and (\ref{eqv1}).

\section{{ Lower bounds for $A$-numerical radius}}
\noindent
We begin with the observation that any $T\in \mathcal{B}_A(\mathcal{H})$ can be expressed as $T=\Re_A(T)+{\rm i} \Im_A(T)$,  where  $\Re_A(T)=\frac{T+T^{\sharp_A}}{2}$ and $\Im_A(T)=\frac{T-T^{\sharp_A}}{2{\rm i}}$.
It is easy to verify that $\Re_A(T)$ and $\Im_A(T)$ both are $A$-self-adjoint, i.e.,
$A\Re_A(T)=(\Re_A(T))^*A$ and $A\Im_A(T)=(\Im_A(T))^*A$. Therefore,  $w_A(\Re_A(T))=\|\Re_A(T)\|_A$ and $w_A(\Im_A(T))=\|\Im_A(T)\|_A$. Now, we are in a position to prove our first improvement.

\begin{theorem}\label{th1}
If $T\in \mathcal{B}_A(\mathcal{H})$, then
\[w_A(T)\geq \frac{\|T\|_A}{2}+\frac{| \|\Re_A(T)\|_A-\|\Im_A(T)\|_A|}{2}.\]	
\end{theorem}
\begin{proof}
	Let $x\in \mathcal{H}$ with $\|x\|_A=1.$ Then from $T=\Re_A(T)+{\rm i} \Im_A(T)$, we have
	$$|\langle Tx,x\rangle_A|^2=|\langle \Re_A(T)x,x\rangle_A|^2+|\langle \Re_A(T)x,x\rangle_A|^2.$$
	This implies that 
	\[|\langle Tx,x\rangle_A|\geq |\langle \Re_A(T)x,x\rangle_A|\,\,\text{and}\,\, |\langle Tx,x\rangle_A|\geq |\langle \Im_A(T)x,x\rangle_A|.\]
	Considering supremum over $\|x\|_A=1,$ we get
	\[w_A(T)\geq \|\Re_A(T)\|_A\,\,\text{and}\,\, w_A(T)\geq \|\Im_A(T)\|_A.\]
	Hence,
	\begin{eqnarray*}
	w_A(T)&\geq& \max \{\|\Re_A(T)\|_A,\|\Im_A(T)\|_A\}\\
	&=& \frac{\|\Re_A(T)\|_A+\|\Im_A(T)\|_A}{2}+\frac{|\|\Re_A(T)\|_A-\|\Im_A(T)\|_A|}{2}\\
	&\geq& \frac{\|\Re_A(T)+\rm i\Im_A(T)\|_A}{2}+\frac{|\|\Re_A(T)\|_A-\|\Im_A(T)\|_A|}{2}\\
	&=&\frac{\|T\|_A}{2}+\frac{| \|\Re_A(T)\|_A-\|\Im_A(T)\|_A|}{2}.
	\end{eqnarray*}
Thus, we complete the proof.
\end{proof}

\begin{remark}
Clearly, the inequality in Theorem \ref{th1} is sharper than the first inequality in (\ref{eqv}), i.e.,  $w_A(T)\geq \frac{\|T\|_A}{2}$.
\end{remark}

In the next theorem we provide a characterzation for the equality of lower bound of $A$-numerical radius mentioned in  (\ref{eqv}).
\begin{theorem}\label{equality1}
Let $T\in \mathcal{B}_A(\mathcal{H})$. \\
(i) If $w_A(T)= \frac{\|T\|_A}{2}$, then	
\[\|\Re_A(T)\|_A=\|\Im_A(T)\|_A=\frac{\|T\|_A}{2},\]
but the converse is not necessarliy true.\\
(ii) $w_A(T)= \frac{\|T\|_A}{2}$ if and only if $\|\Re_A(e^{\rm i \theta}T)\|_A=\|\Im_A(e^{\rm i \theta}T)\|_A=\frac{\|T\|_A}{2}$ for all $\theta \in \mathbb{R}.$	
\end{theorem}
\begin{proof}
	(i) It follows from Theorem \ref{th1} that if $w_A(T)= \frac{\|T\|_A}{2}$, then	
	$$\|\Re_A(T)\|_A=\|\Im_A(T)\|_A.$$ Also, we get
		\begin{eqnarray*}
			\|\Re_A(T)\|_A\leq w_A(T)=\frac{\|T\|_A}{2}=\frac{\|\Re_A(T)+\rm i\Im_A(T)\|_A}{2}&\leq& \frac{\|\Re_A(T)\|_A+\|\Im_A(T)\|_A}{2}\\
			&=&\|\Re_A(T)\|_A.
		\end{eqnarray*}
	This implies that $\|\Re_A(T)\|_A=\frac{\|T\|_A}{2}$, and so  $\|\Im_A(T)\|_A=\frac{\|T\|_A}{2}.$\\
	
\noindent (ii)    The ``if" part follows from $w_A(T)=\sup_{\theta\in \mathbb{R}}\|\Re_A(e^{\rm i \theta}T)\|_A$, and so we only need to prove the ``only if" part. Let $w_A(T)= \frac{\|T\|_A}{2}$. Clearly $e^{\rm i \theta}T\in \mathcal{B}_A(\mathcal{H})$  for all $\theta \in \mathbb{R}.$	Now, $w_A(e^{\rm i \theta}T)=w_A(T)$ and $\|e^{\rm i \theta}T\|_A=\|T\|_A$  for all $\theta \in \mathbb{R}.$ Therefore, it follows from (i)  that  $\|\Re_A(e^{\rm i \theta}T)\|_A=\|\Im_A(e^{\rm i \theta}T)\|_A=\frac{\|T\|_A}{2}$ for all $\theta \in \mathbb{R}.$
\end{proof}

Our next improvement of the first inequality in (\ref{eqv1}) reads as follows.

\begin{theorem}\label{th2}
	If $T \in \mathcal{B}_A(\mathcal{H})$, then
	\begin{eqnarray*}
		  w_A(T) &\geq&  \sqrt {\frac{1}{4} \left \| T^{\sharp_A}T+TT^{\sharp_A} \right\|_A + \frac{1}{2}\left|~\|\Re_A(T)\|_A^2-\|\Im_A(T)\|_A^2~~ \right|}.
	\end{eqnarray*}
\end{theorem}

\begin{proof}
We have $ w_A(T)\geq \|\Re_A(T)\|_A\,\,\text{and}\,\, w_A(T)\geq \|\Im_A(T)\|_A  $ and so 
	\begin{eqnarray*}
		w^2_A(T)&\geq& \max \left\{\| \Re_A(T) \|^2_A,  \| \Im_A(T) \|_A^2  \right\}\\
		&=& \frac{\| \Re_A(T) \|_A^2+  \| \Im_A(T) \|_A^2}{2}+ \frac{  \left| \| \Re_A(T) \|_A^2- \| \Im_A(T) \|_A^2\right|}{2} \\
		&\geq & \frac{\|(\Re_A(T))^2 \|_A+\|(\Im_A(T))^2\|_A}{2}+ \frac{  \left| \| \Re_A(T) \|_A^2- \| \Im_A(T) \|_A^2\right|}{2} \\
		&\geq& \frac{\|(\Re_A(T))^2 +(\Im_A(T))^2  \|_A}{2}+ \frac{  \left| \| \Re_A(T) \|_A^2- \| \Im_A(T) \|_A^2\right|}{2} \\
		&=& {\frac{1}{4} \left \| T^{\sharp_A}T+TT^{\sharp_A} \right\|_A + \frac{1}{2}\left|~\|\Re_A(T)\|_A^2-\|\Im_A(T)\|_A^2~~ \right|}.
	\end{eqnarray*}
This completes the proof.

\end{proof}

\begin{remark}
Clearly, the inequality in Theorem \ref{th2} is sharper than the first inequality in (\ref{eqv1}), i.e., $w^2_A(T)$ $\geq \frac{1}{4} \left \| T^{\sharp_A}T+TT^{\sharp_A} \right\|_A$.
\end{remark}

 Next, we prove an equivalent condition for $w_A(T)=\sqrt{\frac{1}{4}{\left\| T^{\sharp_A}T+TT^{\sharp_A}\right\|_A}}$.

\begin{theorem}\label{equality2}
	Let $T\in \mathcal{B}_A(\mathcal{H})$. Then, $w_A(T)=\sqrt{\frac{1}{4}{\left\| T^{\sharp_A}T+TT^{\sharp_A}\right\|_A}}$ if and only if $\|\Re_A(e^{\rm i \theta}T)\|^2_A=\|\Im(e^{\rm i \theta}T)\|_A^2=\frac{1}{4}\| T^{\sharp_A}T+TT^{\sharp_A}\|_A$ for all $\theta \in \mathbb{R}$.
\end{theorem}

\begin{proof}
	The ``if" part is trivial,  we only prove the ``only if" part. Let $w^2_A(T)=\frac{1}{4}\| T^{\sharp_A}T+TT^{\sharp_A}\|_A.$ Now,  $\left(\Re_A(e^{\rm i \theta}T)\right )^2+\left(\Im_A(e^{\rm i \theta}T)\right)^2=\frac{ T^{\sharp_A}T+TT^{\sharp_A}}{2}$ for all $\theta\in \mathbb{R}$. Therefore, we have
	\begin{eqnarray*}
		\frac{1}{4}\| T^{\sharp_A}T+TT^{\sharp_A}\|_A &=& \frac{1}{2} \left\|\left(\Re_A(e^{\rm i \theta}T)\right )^2+\left(\Im_A(e^{\rm i \theta}T)\right)^2 \right\|_A \\
		&\leq& \frac{1}{2} \left( \left\| \Re_A(e^{\rm i \theta}T) \right\|_A ^2+\left\| \Im_A(e^{\rm i \theta}T)\right\|_A ^2 \right)\\
		&\leq& w^2_A(T)
		= \frac{1}{4}\| T^{\sharp_A}T+TT^{\sharp_A}\|_A.
	\end{eqnarray*}
	Hence,  $ \left\| \Re_A(e^{\rm i \theta}T) \right\|_A ^2+\left\| \Im_A(e^{\rm i \theta}T)\right\|_A ^2 = \frac{1}{2}\| T^{\sharp_A}T+TT^{\sharp_A}\|_A.$
	Now, $\sup_{\theta\in \mathbb{R}}\left\| \Re_A(e^{\rm i \theta}T) \right\|_A ^2=\frac{1}{4}\| T^{\sharp_A}T+TT^{\sharp_A}\|_A=\sup_{\theta\in \mathbb{R}}\left\| \Im_A(e^{\rm i \theta}T) \right\|_A ^2.$ Therefore, $\|\Re_A(e^{\rm i \theta}T)\|^2_A=\|\Im_A(e^{\rm i \theta}T)\|^2_A$ $=\frac{1}{4}\| T^{\sharp_A}T+TT^{\sharp_A}\|_A$ for all $\theta \in \mathbb{R}$.
\end{proof}

Now we obtain another characterizations for the equalities  $w_A(T)=\frac{1}{2}\|T\|_A$ and  $w_A(T)=\sqrt{\frac{ 1  }{4}{\|T^{\sharp_A}T+TT^{\sharp_A}\|_A} }$. First we need to prove the following lemma.

\begin{lemma}\label{lem2}
	Let $T\in \mathcal{B}_A(\mathcal{H})$. Then, 
	 $\|\Re_A(e^{\rm i \theta }T)\|_A=k$ (i.e., a constant) for all $\theta \in \mathbb{R}$ if and only if
	 ${W_A(T)}$ is a circular disk with center at the origin and radius $k$.
\end{lemma}

\begin{proof}
	The ``if" part is trivial, we only prove the ``only if" part.
Let $\left\|\Re_A(e^{\rm i \theta }T)\right \|_A$ $=k$ for all $\theta \in \mathbb{R}$.
	Then, $\sup_{\|x\|_A=1}|\langle \Re_A(e^{\rm i \theta }T)x,x\rangle_A |=k$ for all $\theta \in \mathbb{R}$, i.e., $\sup_{\|x\|_A=1}|Re (e^{\rm i \theta }\langle Tx,x\rangle_A )|=k$ for all $\theta \in \mathbb{R}$. Thus, for each $\theta \in \mathbb{R}$, there exists a sequence $\{x_n^{\theta} \} \subseteq \mathcal{H}$  with $\|x_n^{\theta}\|_A=1$ such that $|Re (e^{\rm i \theta }\langle Tx_{n}^{\theta},x_{n}^{\theta}\rangle_A )|\to k$. This implies that the boundary of  $W_A(T)$ must be a circle with center at the origin and radius $k$. Since $W_A(T)$ is a convex subset of $\mathbb{C}$ (see in \cite[Th. 2.1]{BFS}), so ${W_A(T)}$ is a circular disk with center at the origin and radius $k$.

\end{proof}

\begin{theorem}\label{equality3}
	Let $T\in \mathcal{B}_A(\mathcal{H})$. Then, the following results hold. \\
	(i) $w_A(T)=\frac{1}{2}\|T\|_A$ if and only if ${W_A(T)}$ is a circular disk with center at the origin and radius $\frac{1}{2}\|T\|_A$.\\
	(ii) $w_A(T)=  \sqrt{ \frac{ 1}{4}  \|T^{\sharp_A}T+TT^{\sharp_A}\|_A  }$ if and only if ${W_A(T)}$ is a circular disk with center at the origin and radius $  \sqrt{ \frac{ 1}{4}  \|T^{\sharp_A}T+TT^{\sharp_A}\|_A  }$.
\end{theorem}

\begin{proof}
The proof of (i) and (ii) follow from Theorems \ref{equality1} and \ref{equality2}, respectively, by using Lemma \ref{lem2}. 
\end{proof}

Another improvement of the first inequality in (\ref{eqv}) reads as:

\begin{theorem}\label{th3}
	If $ T\in\mathcal{B}_A(\mathcal{H})$, then 
	$$ w_A(T)\geq\frac{ \|T\|_A}{2}+\frac{ \mid \|\Re_A(T)+\Im_A(T)\|_A-\|\Re_A(T)-\Im_A(T)\|_A  \mid}{2\sqrt{2}}.  $$ 
\end{theorem}
\begin{proof}
	Let $x \in \mathcal{H}$ with $\|x\|_A=1$. Then, we have
	\begin{eqnarray*}
		\mid\langle Tx,x\rangle_A\mid &=& \sqrt{|\langle \Re_A(T)x,x\rangle_A|^2+|\langle \Im_A(T)x,x\rangle_A|^2}\\
		&\geq& \frac{1}{\sqrt{2}}(\mid\langle \Re_A(T)x,x\rangle_A\mid+\mid\langle \Im_A(T)x,x\rangle\mid)\\ &\geq&\frac{1}{\sqrt{2}}\mid\langle (\Re_A(T)\pm  \Im_A(T))x,x\rangle_A\mid.
	\end{eqnarray*}
	Taking supremum over  $\|x\|_A=1$, we get 
	$$ w_A(T)\geq\frac{1}{\sqrt{2}}\|\Re_A(T)\pm \Im_A(T)\|_A. $$ 
	Therefore, we have 
	\begin{eqnarray*}
		w_A(T)&\geq&\frac{1}{\sqrt{2}}\max\{\|\Re_A(T)+\Im_A(T)\|_A,\|\Re_A(T)-\Im_A(T)\|_A\}\\
		&=&\frac{1}{\sqrt{2}}  \Big  \lbrace\frac{ \|\Re_A(T)+\Im_A(T)\|_A+\|\Re_A(T)-\Im_A(T)\|_A}{2}\\
		&&+\frac{\mid \|\Re_A(T)+\Im_A(T)\|_A-\|\Re_A(T)-\Im_A(T)\|_A\mid}{2}   \Big \rbrace\\ 
		&\geq&\frac{1}{\sqrt{2}}   \Big \lbrace\frac{ \|(\Re_A(T)+\Im_A(T))-\rm i(\Re_A(T)-\Im_A(T))\|_A}{2}\\
		&& +\frac{\mid \|\Re_A(T)+\Im_A(T)\|_A-\|\Re_A(T)-\Im_A(T)\|_A\mid}{2}  \Big \rbrace\\
		&=&\frac{1}{\sqrt{2}}\left\lbrace\frac{ \|(1-\rm i)T\|_A}{2}+\frac{\mid \|\Re_A(T)+\Im_A(T)\|_A-\|\Re_A(T)-\Im_A(T)\|_A\mid}{2}\right\rbrace\\
		&=&\frac{\|T\|_A}{2}+\frac{\mid \|\Re_A(T)+\Im_A(T)\|_A-\|\Re_A(T)-\Im_A(T)\|_A\mid}{2\sqrt{2}},
	\end{eqnarray*} 
	as desired.
\end{proof}

\begin{remark}
(i) Clearly, the inequality in Theorem \ref{th3} is sharper than  the first inequality in (\ref{eqv}), i.e., $w_A(T)\geq \frac{\|T\|_A}{2}$.\\
(ii) The inequalities obtained in Theorem \ref{th1} and Theorem \ref{th3} are not comparable, in general.

\end{remark}

Another refinement of the first inequality in (\ref{eqv1}) reads as follows:

\begin{theorem}\label{th4}
	If $ T\in\mathcal{B}_A(\mathcal{H})$, then 
	$$ w_A(T)\geq   \sqrt{ \frac{\| T^{\sharp_A}T+TT^{\sharp_A}\|_A}{4}+\frac{\mid \|\Re_A(T)+\Im_A(T)\|^2_A-\|\Re_A(T)-\Im_A(T)\|^2_A\mid}{4}. }$$
	\label{thnn}\end{theorem}

\begin{proof}
	Following the proof of Theorem \ref{th3}, we have
	\begin{eqnarray*}
		w^2_A(T)&\geq&	\frac{1}{2}\max\{\|\Re_A(T)+\Im_A(T)\|^2_A,\|\Re_A(T)-\Im_A(T)\|^2_A\}\\
		&=& \frac{1}{2}\Big \lbrace\frac{ \|\Re_A(T)+\Im_A(T)\|^2_A+\|\Re_A(T)-\Im_A(T)\|^2_A}{2}\\
		&&+\frac{\mid \|\Re_A(T)+\Im_A(T)\|^2_A-\|\Re_A(T)-\Im_A(T)\|^2_A\mid}{2}\Big\rbrace\\
		&\geq&\frac{1}{2}\Big \lbrace\frac{ \|(\Re_A(T)+\Im_A(T))^2+(\Re_A(T)-\Im_A(T))^2\|_A}{2}\\
		&&+\frac{\mid \|\Re_A(T)+\Im_A(T)\|^2_A-\|\Re_A(T)-\Im_A(T)\|^2_A\mid}{2}\Big\rbrace\\ &=&\frac{\| T^{\sharp_A}T+TT^{\sharp_A}\|_A}{4}+\frac{\mid \|\Re_A(T)+\Im_A(T)\|^2_A-\|\Re_A(T)-\Im_A(T)\|^2_A\mid}{4}.
	\end{eqnarray*}
	This completes the proof.
\end{proof}

\begin{remark}
(i) 	Clearly, the inequality in Theorem \ref{th4} is sharper than the first inequality in (\ref{eqv1}), i.e.,  $w^2_A(T)\geq \frac{1}{4} \left \| T^{\sharp_A}T+TT^{\sharp_A} \right\|_A$. \\
(ii)  The inequalities obtained in Theorem \ref{th2} and Theorem \ref{th4} are not comparable, in general.
\end{remark}

\section{Applications}

\noindent In this section we obtain new inequalities for the $A$-numerical radius of the generalized commutators of operators by applying Theorems \ref{th2} and \ref{th4}. First we prove the following lemma.

\begin{lemma}\label{lem1}
If $ T,X,Y\in\mathcal{B}_A(\mathcal{H})$, then
\begin{eqnarray*}
	w_A(TX\pm YT)\leq \max \left\{\|X\|_A,\|Y\|_A \right\} \sqrt{2\|T^{\sharp_A}T+TT^{\sharp_A}\|_A}.
\end{eqnarray*}
\end{lemma}

\begin{proof}
Let $x\in {\mathcal{H}}$ with $\|x\|_A=1$ and  $\max \{ \|X\|_A,\|Y\|_A \} \leq 1$. Then by Cauchy Schwarz inequality, we get 
\begin{eqnarray*}
	|\langle (TX\pm YT)x,x\rangle_A|
	&\leq&  |\langle Xx,T^{\sharp_A}x\rangle_A|+|\langle Tx,Y^{\sharp_A}x\rangle_A| \\
	&\leq& \|T^{\sharp_A}x\|+ \|Tx\| \\ 
	&\leq &  \sqrt{2}\left (\|T^{\sharp_A}x\|_A^2+ \|Tx\|_A^2\right )^{\frac{1}{2}}\\
	&\leq&  \sqrt{2}\|T^{\sharp_A}T+TT^{\sharp_A}\|_A^{\frac{1}{2}}.
\end{eqnarray*}
Therefore, taking supremum over $\|x\|_A=1$, we get 
\begin{eqnarray*}
w_A(TX\pm YT)&\leq& \sqrt{2\|T^{\sharp_A}T+TT^{\sharp_A}\|_A}.
\end{eqnarray*}
If  $ X,Y\in \mathcal{B}_A(\mathcal{H})$ are arbitrary with $\max \{ \|X\|_A,\|Y\|_A \} \neq 0$, then it follows from the above inequality  that 
\begin{eqnarray*}
w_A(TX\pm YT)\leq \max  \left\{\|X\|_A,\|Y\|_A \right\}   \sqrt{2\|T^{\sharp_A}T+TT^{\sharp_A}\|_A}.
\end{eqnarray*}
Also, if $\max \{ \|X\|_A,\|Y\|_A \} = 0$, then the above inequality holds trivially.
This completes the proof.
\end{proof}

\begin{theorem}\label{th5}
	If $T,X,Y \in \mathcal{B}_A(\mathcal{H})$, then
\begin{eqnarray*}
	(i)\,\, w_A(TX\pm YT)\leq 2\sqrt{2}\max  \left\{\|X\|_A,\|Y\|_A \right\}\sqrt{w_A^2(T)- \frac{\left|~\|\Re_A(T)\|_A^2-\|\Im_A(T)\|_A^2~~ \right|}{2}}.
\end{eqnarray*}
and \\
(ii)\, $ w_A(TX\pm YT) $
 $$\leq 2\sqrt{2}\max  \left\{\|X\|_A,\|Y\|_A \right\}\sqrt{w_A^2(T)- \frac{\mid \|\Re_A(T)+\Im_A(T)\|^2_A-\|\Re_A(T)-\Im_A(T)\|^2_A\mid}{4}   }.$$
 \end{theorem}
\begin{proof}
	By applying the inequalities in Theorem \ref{th2} and Theorem \ref{th4} in Lemma \ref{lem1}, we have (i) and (ii), respectively.
\end{proof}

It should be mentioned here that the inequalities (i) and (ii) in Theorem \ref{th5} are not comparable, in general.

Considering $X=Y=S$ in Theorem \ref{th5}, we get the following corollary.

\begin{cor}\label{cor5}
If $T,S \in \mathcal{B}_A(\mathcal{H})$, then
\begin{eqnarray*}
(i)\,\, w_A(TS\pm ST)\leq 2\sqrt{2} \|S\|_A \sqrt{w_A^2(T)- \frac{\left|~\|\Re_A(T)\|_A^2-\|\Im_A(T)\|_A^2~~ \right|}{2}}.
\end{eqnarray*}
and
\begin{eqnarray*}
 (ii)&& w_A(TS\pm ST)\\
& & \leq  2\sqrt{2} \|S\|_A \sqrt{w_A^2(T)- \frac{\mid \|\Re_A(T)+\Im_A(T)\|^2_A-\|\Re_A(T)-\Im_A(T)\|^2_A\mid}{4}  }.
\end{eqnarray*}
	
\end{cor}

\noindent Now, interchanging $T$ and $S$ in Corollary \ref{cor5} (i), we get 
\begin{eqnarray}\label{1}
w_A(TS\pm ST) &\leq& 2\sqrt{2} \min \{\alpha_1,\alpha_2\},
\end{eqnarray}
where 
\begin{eqnarray*}
\alpha_1&=& \|S\|_A \sqrt{w_A^2(T)- \frac{\left|~\|\Re_A(T)\|_A^2-\|\Im_A(T)\|_A^2~~ \right|}{2}},\\
\alpha_2&=& \|T\|_A \sqrt{w_A^2(S)- \frac{\left|~\|\Re_A(S)\|_A^2-\|\Im_A(S)\|_A^2~~ \right|}{2}}.
\end{eqnarray*}

\noindent Also, interchanging $T$ and $S$ in Corollary \ref{cor5} (ii), we get 
\begin{eqnarray}\label{2}
w_A(TS\pm ST) &\leq& 2\sqrt{2} \min \{\beta_1,\beta_2\},
\end{eqnarray}
where 
\begin{eqnarray*}
	\beta_1&=& \|S\|_A \sqrt{w_A^2(T)- \frac{\mid \|\Re_A(T)+\Im_A(T)\|^2_A-\|\Re_A(T)-\Im_A(T)\|^2_A\mid}{4}  },\\
	\beta_2&=& \|T\|_A \sqrt{w_A^2(S)- \frac{\mid \|\Re_A(S)+\Im_A(S)\|^2_A-\|\Re_A(S)-\Im_A(S)\|^2_A\mid}{4}  }.
\end{eqnarray*}

\begin{remark}
	In \cite[Th. 4.2]{Zamani}, Zamani proved that if $T,S \in \mathcal{B}_A(\mathcal{H})$, then
	\begin{eqnarray*}
	w_A(TS\pm ST) &\leq& 2\sqrt{2} \min \{\|T\|_A w_A(S), \|S\|_A w_A(T)    \}.
	\end{eqnarray*}
	Clearly, both the inequalities in (\ref{1}) and (\ref{2}) are stronger than the inequality in \cite[Th. 4.2]{Zamani}.

\end{remark}

\noindent \textbf{Data Availability Statement.}\\
Data sharing not applicable to this article as no datasets were generated or analysed during the current study.

\bibliographystyle{amsplain}

\end{document}